\documentclass[a4paper,11pt,english,leqno]{amsart}
\usepackage[utf8x]{inputenc}
\usepackage[bitstream-charter]{mathdesign}
\usepackage{amsmath, xcolor,amsthm,epsfig,epstopdf,url,array}
\usepackage[colorlinks=true]{hyperref}
\usepackage[all]{xy}
\usepackage{tikz}
\tikzset{node distance=2cm, auto}
\usepackage{subcaption}
\usepackage[backrefs]{amsrefs}
\usepackage{pgf,tikz}
\usetikzlibrary{arrows}
\usepackage{csquotes}

\theoremstyle{plain}
\newtheorem{thm}{Theorem}[section]
\newtheorem{lem}[thm]{Lemma}
\newtheorem{prop}[thm]{Proposition}
\newtheorem{cor}[thm]{Corollary}

\newtheorem{obs}[thm]{Remark}
\newtheorem{problem}[thm]{Problem}

\theoremstyle{definition}

\newtheorem{prob}[thm]{Problem}

\theoremstyle{remark}

\theoremstyle{plain}

\newcommand{\C}{\mathbb{C}}

\newcommand{\Pc}{\mathbb{P}(\mathbb{C}^2)}
\renewcommand{\S}{\mathbb{S}^2}

\newcommand{\CH}{\ \mathcal{H}_N}

\begin{document}

\title[A sharp Bombieri inequality]{A sharp Bombieri inequality, logarithmic energy and well conditioned polynomials}

\author{Uju\'e Etayo}

\date{\today{}}

\thanks{The author has been supported by the Austrian Science Fund FWF project F5503 (part of the Special Research Program (SFB) Quasi-Monte Carlo Methods: Theory and Applications), by MTM2017-83816-P from Spanish Ministry of Science MICINN and by 21.SI01.64658 from Universidad de Cantabria and Banco de Santander.
}



%\author{Uju\'e Etayo}
\address{5010 Institute of Analysis and Number Theory 8010 Graz, Kopernikusgasse 24/II}
\email{etayo@math.tugraz.at}

\begin{abstract}
In this paper we explore the connections between minimizers of the discrete logarithmic energy on the sphere $\mathbb{S}^{2}$, univariate polynomials with optimal condition number in the Shub-Smale sense and a quotient involving norms of polynomials.
Our main results are that polynomials with optimal condition number produce spherical points with small logarithmic energy (the reverse result was already known) %, we provide a new characterization for the minimal logarithmic energy on the sphere 
and a sharp Bombieri type inequality for univariate polynomials with complex coefficients.
\end{abstract}

\maketitle
\tableofcontents

%------------------------------------------------------------
%------------------------------------------------------------
%------------------------------------------------------------
%------------------------------------------------------------

%------------------------------------------------------------
%-------------------------------------it-----------------------
%------------------------------------------------------------
%------------------------------------------------------------

\section{Introduction and main results}\label{sec1}

%------------------------------------------------------------
%------------------------------------------------------------
%------------------------------------------------------------
%------------------------------------------------------------

\subsection{Inequalities of polynomials}

Let $P_{1},\ldots,P_{m}\in K[x_{1},\ldots,x_{n}]$ be a set of polynomials with $K = \mathbb{C}$ or $\mathbb{R}$ and let $||*||$ be a norm defined in $K[x_{1},\ldots,x_{n}]$.
%A frequent question in the literature is given in the following problem.
\begin{prob}\label{prob1}\label{eq:main}
Define a constant $M$ depending only on the degrees of the polynomials $P_{1},\ldots,P_{m}$ and such that
\begin{equation*}
||P_{1}||\ldots||P_{m}||
\leq
M
||P_{1}\ldots P_{m}||.
\end{equation*}
\end{prob}
The literature versing on this problem is rather extensive, we refer the interested reader to \cites{BEAUZAMY1990219,borwein_1994,BOYD1993115,10.1112/blms/26.5.449,PINASCO,BEAUZAMY1985390,Pritsker_Ruscheweyh_2009,Pritsker2009} for example.
The inequality in Problem \ref{eq:main} has been proposed for many different norms, we recomend %the surveys  on the topic given in \cites{BEAUZAMY1990219,PINASCO} and 
the articles \cites{Pritsker2009,Pritsker_Ruscheweyh_2009} for results relating different norms of polynomials.
Here we are concerned only by the Bombieri-Weyl norm.

%------------------------------------------------------------
%------------------------------------------------------------
%------------------------------------------------------------
%------------------------------------------------------------

\subsubsection{Bombieri-Weyl norm}

We denote by $\CH$ the vector space of bivariate homogeneous polynomials of degree $N$,  that is the set of polynomials of the form
\begin{equation}\label{eq:polyg}
g(x,y)=\sum_{i=0}^Na_i x^iy^{N-i},\quad a_i\in\C
\end{equation}
where $x,y$ are complex variables. The Weyl norm of $g$  (sometimes called  Kostlan or Bombieri-Weyl or Bombieri norm) is
\begin{equation*}
\|g\|=\left(\sum_{i=0}^N\binom{N}{i}^{-1}|a_i|^2\right)^{1/2},
\end{equation*}
where the binomial coefficients satisfy the property $\|g\|=\|g\circ U\|$ where $U\subseteq\C^{2\times 2}$ is any unitary $2\times 2$ matrix and $g\circ U\in\CH$ is the polynomial given by $g\circ U(x,y)=g\left( U\binom{x}{y}\right)$. 
An alternative definition for the Bombieri-Weyl norm is
\begin{equation*}
\|g\|^2=\frac{N+1}{\pi}\int_{ \Pc}\frac{|g(\eta)|^2}{\|\eta\|^{2N}}\,dV(
\eta),
\end{equation*}
where the integration is made with respect to the volume form $V$ 
arising from the standard Riemannian structure in $\Pc$. 
%Note that the expression inside the integral is well defined since it does not depend on the choice of the representative of $\eta\in\Pc$.
%In \cite{Pritsker_Ruscheweyh_2009} and \cite{Pritsker2009} the author proves some inequalities relating the Bombieri norm with the norm on the circle.
Given a univariate degree $N$ polynomial with complex coefficients $P(z)=\sum_{i=0}^Na_iz^i$, it has a homogeneous counterpart $g(x,y)=\sum_{i=0}^Na_i x^iy^{N-i}$ and its Weyl norm is defined via its homogenized version:
\begin{equation*}
\|P\|=\|g\|.
\end{equation*}

%------------------------------------------------------------
%------------------------------------------------------------
%------------------------------------------------------------
%------------------------------------------------------------

\subsubsection{The product of two polynomials}

A partial answer to Problem \ref{eq:main} was given in \cite{BEAUZAMY1990219}, where authors proved what is known as Bombieri's inequality.

\begin{thm}[\cite{BEAUZAMY1990219}*{Theorem 1.2}, Bombieri's inequality]\label{thmBombi}
Let $P,Q$ be homogeneous polynomials of degrees $m,n$ respectively.
Then
\begin{equation*}
||PQ||
\geq
\sqrt{\frac{m!n!}{(m+n)!}}||P|| \,  ||Q||,
\end{equation*}
where $||*||$ is the Bombieri-Weyl norm and the inequality is sharp.
\end{thm}

%From Theorem \ref{thmBombi} we can deduce an inequality of the form of equation \eqref{eq:main}, that is presented for example in:

\begin{cor}[\cite{PINASCO}*{Corollary 2.2}]\label{corBombi}
Let $P_{1},\ldots,P_{m}$ be homogeneous polynomials of degrees $k_{i}$ respectively.
Then
\begin{equation*}
||P_{1},\ldots,P_{m}||
\geq
\sqrt{\frac{k_{1}!\ldots k_{m}!}{(k_{1}+\ldots +k_{m})!}}||P_{1}|| \ldots ||P_{m}||,
\end{equation*}
where $||*||$ is the Bombieri-Weyl norm.
\end{cor}

In \cite{BOYD1993115} the same problem is addresed providing a sharp bound in the particular case of having two polynomials $P$ and $Q$ such that $P$ divides $Q$.

%------------------------------------------------------------
%------------------------------------------------------------
%------------------------------------------------------------
%------------------------------------------------------------

\subsection{A natural inequality}

A natural question in this setting is the following: given a univariate polynomial of degree $N$, that we denote by $p_{N}(x)$, what is the relation between the norm of the polynomial and the product of the norms of its factors? 
We restrict the question to monic polynomials.%That we state in Problem \ref{mainprob}.
\begin{prob}\label{mainprob}
Let $p_{N}(x) = \prod_{i=1}^{N} (x-z_{i})$, with $z_{i} \in \mathbb{C}$ for all $1\leq i \leq N$.
Give a constant $M$ depending only on $N$ such that
\begin{equation*}
\displaystyle \prod_{i=1}^{N} ||x-z_{i}||
\leq
M
\left|\left| \prod_{i=1}^{N} (x-z_{i}) \right|\right|,
\end{equation*}
where $||*||$ is the Bombieri-Weyl norm.
\end{prob}
In \cite{10.1112/blms/26.5.449} the author proves a sharp inequality 
\begin{equation*}
|P_{1}|\ldots|P_{m}|
\leq
C_{m}^{N}
|P_{1}\ldots P_{m}|,
\end{equation*}
where $N$ is the degree of $P_{1}\ldots P_{m}$, $C_m$ is a constant depending on $m$ and $|*|$ is the maximum norm on the unit circle.
To prove this result, he makes use of the logarithmic potential in the complex plane on a different way that we are doing here.

\noindent For the Bombieri-Weyl norm the only answer we have comes from Corollary \ref{corBombi}:
\begin{equation}\label{prebound}
\prod_{i=1}^{N} ||x-z_{i}||
\leq
\sqrt{N!}
\left|\left| \prod_{i=1}^{N} (x-z_{i}) \right|\right|.
\end{equation}
In \cite{PINASCO} the author proves another inequality, see \cite{PINASCO}*{Theorem 4.1} wich generalizes the one proposed in Theorem \ref{thmBombi}, but that in the conditions of Problem \ref{mainprob} provides the same result that equation \eqref{prebound}.

%\cite{benitez_sarantopoulos_tonge_1998} proves that

%------------------------------------------------------------
%------------------------------------------------------------
%------------------------------------------------------------
%------------------------------------------------------------

\subsection{Notation}\label{notation}

Throughout this article we work with points in the complex plane, $\mathbb{C}$, the Riemann sphere that we will denote by $\mathbb{S}\subset\mathbb{R}^{3}$ and is defined as the sphere centered in $(0,0,\frac{1}{2})$ of radius $\frac{1}{2}$, and the unit sphere $\mathbb{S}^2\subset\mathbb{R}^{3}$ centered in $(0,0,0) $ of radius $1$.

Three maps link these three objects.
Let $h:\mathbb{S}\longrightarrow\mathbb{S}^{2}$ be the homothetic transformation that maps $\hat{z}_{i} \mapsto 2\hat{z}_{i} - (0,0,1)$.
Let $\pi_{\mathbb{S}}: \mathbb{S}\backslash\{ (0,0,1) \} \longrightarrow \mathbb{C}$ be the stereographic proyection from the North pole, given by the map
\begin{equation*}
\hat{z} = (\hat{a},\hat{b},\hat{c}) \in \mathbb{S} \mapsto z = \frac{\hat{a} + i\hat{b}}{1-\hat{c}}.
\end{equation*}
Finally, let $\pi_{\mathbb{S}^{2}}: \mathbb{S}^{2}\backslash\{ (0,0,1) \} \longrightarrow \mathbb{C}$ be also an stereographic projection from the North pole given by the map
\begin{equation*}
x = (a,b,c) \in \mathbb{S}^{2} \mapsto z = \frac{a + ib}{1-c}.
\end{equation*}
Note that $(\pi_{\mathbb{S}^{2}} \circ h ) = \pi_{\mathbb{S}}$ and that the inverse applications of $\pi_{\mathbb{S}}$ and $\pi_{\mathbb{S}^{2}}$ are well defined.
%\begin{figure}
%\includegraphics[scale=0.5]{Bombi.png}
%\caption{Stereographic proyection from both spheres $\mathbb{S}$ and $\mathbb{S}^{2}$.}
%\end{figure}

We denote by $z_{i}$ the points on the complex plane $\mathbb{C}$, by $\hat{z}_{i}$ their proyection into the Riemann sphere through  the inverse map of $\pi_{\mathbb{S}}$ and $x_{i}$ their proyection into the unit sphere through the inverse map of $\pi_{\mathbb{S}^{2}}$.

%------------------------------------------------------------
%------------------------------------------------------------
%------------------------------------------------------------
%------------------------------------------------------------

\subsection{Main results}

%------------------------------------------------------------
%------------------------------------------------------------
There are two formulas that relate the logarithmic energy of a set of points on the sphere and  the condition number of a polynomial. % and a quotient of norms of polynomials.
One of them is given in Proposition \ref{def:workingmu} and was used in recent paper \cite{BEMOC} to provide a family of polynomials with small condition number.
The other one is given in Lemma \ref{ABS} and it includes a third ingredient, a quotient of norms of polynomials. % and it has not been exploded so much up to our knowledge.
In this paper we combine both formulas to obtain new results relating these three objects and partially solving Problem \ref{mainprob}. %the logarithmic energy with the condition number and a characterization for Problem \ref{mainprob} that allows us to provide a sharp constant $M$.
We will present here only the main results of the paper, other partial results could be found through sections \ref{sec3} and \ref{sec4}.
\medskip

%------------------------------------------------------------
%------------------------------------------------------------
Our first result is the reverse result of the Main result of \cite{SHUB19934}.
We prove that given a sequence of univariate polynomials with complex coefficients that are well conditioned, then the associated set of points on the sphere $\mathbb{S}^{2}$ has logarithmic energy close to the minimal.
%This is stated as:
\smallskip
%------------------------------------------------------------

\begin{thm}\label{thmain1}
Let $(p_{N}(x))_{N\in \mathbb{N}}$ be a sequence of monic polynomials with optimal condition number, i.e. $\mu_{norm}(p_{N}(x)) \leq C\sqrt{N}$ for some universal constant $C$, let $\omega_{N} = \{ x_{1},\ldots ,x_N \} \subset \mathbb{S}^{2}$ be a set of points obtained through the inverse of the projection $\pi_{\mathbb{S}^{2}}$ of the roots of the polynomials $p_{N}$.
Then 
\begin{equation*}
\mathcal{E}_{\log}(\omega_{N})
\leq
\min\limits_{\mu_{N} \subset \S}\mathcal{E}_{\log} (\mu_{N})
+ cN,
\end{equation*}
with $\mu_{N}$ a $N$-point subset of $\S$ and $c$ a constant depending only on $C$.
\end{thm}

Our second main result is a new inequality comparing the Weyl norm of a polynomial and the product of the norms of every factor of the polynomial.
\smallskip
%------------------------------------------------------------

\begin{thm}\label{thmain2}
Given a set of complex points $z_{1},\ldots,z_{N}$, we have
\begin{equation*}
\displaystyle \prod_{i=1}^{N} ||x-z_{i}||
\leq
\sqrt{\frac{e^{N}}{N+1}}
\left|\left| \prod_{i=1}^{N} (x-z_{i}) \right|\right|,
\end{equation*}
where $||*||$ is the Bombieri-Weyl norm and the bound is sharp up to a constant.%$\lim\limits_{N\longrightarrow\infty} \epsilon = 0$.
\end{thm}

%------------------------------------------------------------
%------------------------------------------------------------

%------------------------------------------------------------
Theorem \ref{thmain2} substantially improves the previously known bounds, stated in equation \eqref{prebound}.
%Not only this, but we prove that the new bound is sharp up to a constant.
%\smallskip
%------------------------------------------------------------

%\noindent \textbf{Theorem \ref{sharp2}}
%\textit{Let $(p_{N}(x))_{N\in \mathbb{N}}$ be a sequence of polynomials defined by $p_{N}(x) = \prod_{i=1}^{N} ||x-z_{i}||$ with optimal condition number, i.e. $\mu_{norm}(p_{N}(x)) \leq C\sqrt{N}$ for some universal constant $C$ for all $N$, then
%\begin{equation*}
%\displaystyle \prod_{i=1}^{N} ||x-z_{i}||
%\geq
%e^{C_{\log} - \log(C)}
%\sqrt{\frac{e^{N}}{N}}
%\left|\left| \prod_{i=1}^{N} (x-z_{i}) \right|\right|,
%\end{equation*}
%where $||*||$ is the Bombieri-Weyl norm.}
%%------------------------------------------------------------
%%------------------------------------------------------------
%
%
%%------------------------------------------------------------
%In Section \ref{sec5} we establish a similar result for minimizers of the logarithmic energy.

%------------------------------------------------------------
%------------------------------------------------------------
%------------------------------------------------------------
%------------------------------------------------------------

\subsection{Organization}

We divide the present article into several thematic sections, where the different results are stated and proved. % and other known results are presented but which, when written in this notation, allow us to prove new properties. 
Section \ref{sec1} contains an introduction to the topic and the main results of the paper.
In Section \ref{sec2} we present some definitions and properties of the condition number of a univariate polynomial with complex coefficients.
In Section \ref{sec3} we prove some results relating the condition number of polynomials with the logarithmic energy of a set of spherical points and we prove Theorem \ref{thmain1}.
Section \ref{sec4} is devoted to the proof on Theorem \ref{thmain2}.
In Section \ref{sec6} we present some comments on different extensions of Theorem \ref{thmain2} and its implications for the open problems presented on the paper.
We conclude with Section \ref{sec7}, a section devoted to the proof of some auxiliary lemmas that we use for the proofs of the main theorems.
%------------------------------------------------------------
%------------------------------------------------------------
%------------------------------------------------------------
%------------------------------------------------------------

\section{The condition number of polynomials}\label{sec2}

Let us consider a homogeneous polynomial $g$ with complex coefficients as defined in equation \eqref{eq:polyg}.
The zeros of $g$ lie naturally in the complex projective space $\Pc$. 
Let  $\zeta$ be a zero of $g$, if the derivative $Dg(\zeta)$ does not vanish, then by the Implicit Function Theorem the zero $\zeta$ of $g$ can be continued in a unique differentiable manner to a zero $\zeta'$ of any sufficiently close polynomial $g'$.
This thus defines locally a {\em solution map} given by $Sol(g')= \zeta'$. 
 
The condition number is the operator norm of the derivative of the solution map, 
\begin{equation*}
\mu(g,\zeta)=|D Sol(g,\zeta)|
\end{equation*}
where the tangent spaces $T_g\CH$ and $T_\zeta\Pc$ are endowed respectively with the Bombieri--Weyl norm and the Fubini--Study metric. 
 In \cite{BezI} it was proved that
\begin{equation*}%\label{def:mu}
\mu(g,\zeta)=\|g\|\,\|\zeta\|^{N-1}|(Dg(\zeta)\mid_{\zeta^{\perp}})^{-1}|,
\end{equation*}
where $Dg(\zeta)$ is just the derivative
\[
Dg(\zeta)=\left(\frac{\partial}{\partial x} g(x,y) \quad  
\frac{\partial}{\partial y} g(x,y)\right)_{(x,y)=\zeta}
\]
and $Dg(\zeta)\mid_{\zeta^{\perp}}$ is the restriction of this derivative to 
the orthogonal complement of $\zeta$ in $\C^2$. If this restriction is not 
invertible, which corresponds to $\zeta$ being a double root of $g$, then by 
definition $\mu(g,\zeta)=\infty$.

Shub and Smale also introduced a normalized version of the condition number 
that, in the case of polynomials, it is simply defined by
\begin{equation*}%\label{def:munorm}
\mu_{\rm 
norm}(g,\zeta)=\sqrt{N}\,\mu(g,\zeta)=\sqrt{N}\,\|g\|\,\|\zeta\|^{N-1}|(Dg(\zeta)\mid_{\zeta^{\perp}})^
{-1}|.
\end{equation*}
The normalized condition number of $g$ (without reference to a particular 
zero) is defined by
\begin{equation}
\mu_{\rm norm}(g)=\max_{\zeta\in \Pc:g(\zeta)=0}\mu_{\rm norm}(g,\zeta).
\end{equation}
Given a univariate degree $N$ polynomial with complex coefficients $P(z)=\sum_{i=0}^Na_iz^i$, it has a homogeneous counterpart $g(x,y)=\sum_{i=0}^Na_i x^iy^{N-i}$. 
The condition number of $P$ is defined via its homogenized version:
\begin{equation*}
\mu_{\rm norm}(P)=\mu_{\rm norm}(g)=\max_{z\in \C:P(z)=0}\mu_{\rm norm}(P,z).
\end{equation*}
A simple expression for the condition number of a univariate polynomial is:
\begin{equation*}%\label{eq:mucomp}
\mu_{\rm norm}(P,z)=N^{1/2}\frac{\|P\| (1+|z|^2)^{N/2-1}}{|P'(z)|},
\end{equation*}
and we have $\mu_{\rm norm}(P,z)=\infty$ if and only if $z$ is a double zero of $P$, see for example \cite{Beltran:2015:FLF:2729553.2729572}.

%Given a point $z\in\mathbb C$ we denote by $\hat z$ the point in $\S = 
%\{(a,b,c)\in\mathbb{R}^3: \ a^2+b^2+c^2 = 1\}$ 
%obtained from the stereographic projection.
%That is if we denote $\hat z = (a, b, c)$ then $z = 
%(a+ib)/(1-c)$ and conversely 
%\[
%a = \frac{z+\bar z}{1+|z|^2},\qquad
%b =  \frac{z-\bar z}{i(1+|z|^2)},\qquad
%c =  \frac{|z|^2-1}{1+|z|^2}.
%\]
\subsection{A characterization for the condition number of univariate polynomials}
In this section we use the maps defined in Section \ref{notation}.
Given a monic polynomial $P(z) = \prod_{i = 1}^N (z-z_i)$ we 
consider the continuous function $\hat P: \mathbb{S}\to \mathbb R$ defined as $\hat P(x) = \prod_{i = 1}^N |x-\hat z_i|$. 
Moreover for any given zero $z$ of $P$ we define $\hat P_z(x) = \hat P(x)/|x-\hat{z}|$, that in the case $x=z=\hat z_i$ for some $i$ simply means
\[
\hat P_{z_i}(\hat z_i)=\prod_{j\neq i}|\hat z_i-\hat z_j|.
\]

%------------------------------------------------------------------
%------------------------------------------------------------------
%In \cite{SHUB19934} authors give the following characterization of the condition number.
\begin{prop}[\cite{SHUB19934}*{Proposition~2}]\label{def:workingmu}
With the previous notation,
\begin{equation*}
\mu_{\rm norm}(P, z) = \frac{\sqrt{N(N+1)}}{\sqrt{\pi}}\frac{\|\hat P\|_{L^2}}{\hat P_z(\hat{z})}.
\end{equation*}
%where $d\sigma$ is the sphere surface measure, normalized to satisfy $\sigma(\S)=1$.
\end{prop}

%------------------------------------------------------------------
%------------------------------------------------------------------

Note that Proposition \ref{def:workingmu} is stated for points on the Riemann sphere $\mathbb{S}$.
In this paper we use the notation of \cite{BEMOC} working with the unit sphere $\mathbb{S}^2$, so that the equation in Proposition \ref{def:workingmu} reads 
%\begin{equation*}
%\mu_{\rm norm}(P, \zeta) = \frac12\sqrt{N(N+1)}\frac{\|\hat P\|_{L^2(d\sigma)}}{\hat P_\zeta(\hat\zeta)}.
%\end{equation*}
%In other words, we have
\begin{equation}\label{def:workingmu1}
\mu_{\rm norm}(P,z_{i}) =\frac12\sqrt{N(N+1)} 
\frac{\left(\int_{\S}\prod_{j=1}^N| p-x_j|^2d\sigma(p)\right)^{1/2}}{\prod_{j\neq i}|x_i-x_j|}
\end{equation}
where $d\sigma$ is the sphere surface measure, normalized to satisfy $\sigma(\S)=1$ and $z_{j} = \pi_{\S} (x_{j})$ are the roots of $P(x)$,
and so,
\begin{equation}\label{def:workingmu2}
\mu_{\rm norm}(P) =\frac12\sqrt{N(N+1)}\max_{1\leq i\leq N} 
\frac{\left(\int_{\S}\prod_{j=1}^N| p-x_j|^2d\sigma(p)\right)^{1/2}}{\prod_{j\neq i}|x_i-x_j|}.
\end{equation}
%------------------------------------------------------------------
%------------------------------------------------------------------
%Now we describe the main result in \cite{SHUB19934}. 

%------------------------------------------------------------
%------------------------------------------------------------
%------------------------------------------------------------
%------------------------------------------------------------

\subsection{Optimal condition number}
%The condition number of a polynomial $P\in\mathbb{C}[x]$ at its root $z$ shows how $z$ changes when we perturb a little bit $P$.
%When looking for stable algorithms, we are interested in working with polynomials with associated small condition number, see \cite{SHUB19934}.
%If we consider for example, the condition number of the polynomial $z^N-1$, it is equal at all of its zeros and
%\begin{equation*}%\label{eq:aicesunidad}
%\mu_{\rm 
%norm}(z^N-1)=N^{1/2}\frac{\|z^N-1\|2^{N/2-1}}{N}=\frac{2^{N/2-1/2}}{\sqrt{N}}.
%\end{equation*}
%(Note that the same computation gives a slightly different result in  \cite[p. 7]{SHUB19934}; the correct quantity is \eqref{eq:aicesunidad}).
The main problem of \cite{SHUB19934} consists on giving explicitly a sequence of polynomials $(p_{N}(x))_{N}$ (where $N$ is the degree of the polynomial) with condition number bounded by $\mu_{norm}(p_{N}(x))\leq CN$, where $C$ is an universal constant.
The problem was recently solved in \cite[Theorem 1.5]{BEMOC}, where the authors provide an explicit sequence of polynomials $(p_{N}(x))_{N}$ with condition number bounded by $\mu_{norm}(p_{N}(x))\leq C\sqrt{N}$ where $C$ is an universal constant.
In the same article authors prove that this bound is optimal, i.e. there exists a universal constant $c$ such that $\mu_{norm}(p_{N}(x))\geq c\sqrt{N}$ for any sequence of polynomials $(p_{N}(x))_{N}$.
%They also proved that this bound is sharp in the following way.
%
%\begin{lem}[Lemma 1.1 in \cite{BEMOC}]\label{condmin}
%There is a universal constant $c$ such that $\mu_{\rm norm}(p)\geq c\sqrt{N}$ for every degree $N$ polynomial $p$.
%\end{lem}
%
%In Proposition \ref{propcondmin} we present a stronger result than Lemma \ref{condmin} and in Proposition \ref{condminforenermin} we give an explicit $c$ for polynomilas coming from minimizers of the logarithmic energy.

%------------------------------------------------------------
%------------------------------------------------------------
%------------------------------------------------------------
%------------------------------------------------------------

\section{The logarithimic energy and the condition number of polynomials}\label{sec3}

For a set of points $\omega_{N} = \{ x_1,\ldots,x_N\}$ in the unit sphere $\S\subseteq\mathbb{R}^3$, the logarithmic energy of $\omega_{N}$ is definded as
\begin{equation}\label{defnener}
\mathcal{E}_{\log}(\omega_N)
=
-
\sum_{i\neq j}\log ||x_i-x_j||.
\end{equation}
%(note that in \cite{SHUB19934} the sum is taken over $i< j$ instead of $i\neq j$, which is equivalent to dividing $\mathcal{E}$ by $2$. Here we follow the notation in most of the current works in the area).
The minimal value of this energy has been largely studied, we recomend the interested reader the recent book \cite{borodachov2019discrete}.
One of the motivations to study the minimal logarithmic energy comes from the following theorem.
%Let
%\[
%\mathcal{E}_N=\min_{\hat z_1,\ldots,\hat z_N\in\S}\mathcal{E}(\hat 
%z_1,\ldots,\hat z_N).
%\]
\begin{thm}[Main result of \cite{SHUB19934}]\label{th:bez3}
Let $\omega_{N} = \{x_1,\ldots,x_N\}\subset\S$ be such that
\begin{equation*}
\mathcal{E}_{\log}(\omega_{N})
\leq 
\min\limits_{\mu_{N} \subset \S}\mathcal{E}_{\log} (\mu_{N}) +c\log N,
\end{equation*}
where $\mu_{N}$ is a set of $N$ points in $\mathbb{S}^{2}$.
Let $z_1,\ldots,z_N$ be the complex points given by the stereographic projection $\pi_{\mathbb{S}^{2}}$ of $x_1,\ldots,x_N$. 
Then, the polynomial $P(x) = \prod_{i=1}^N (x-z_i)$ satisfies $\mu_{\rm norm}(P)\leq \sqrt{N^{1+c}(N+1)}$.
\end{thm}
%We have then the following corollary.
\begin{cor}\label{corBeIII}
Let $\omega_{N} = \{x_1,\ldots,x_N\}\subset\S$ be a set of minimizers of the logarithmic energy, let $z_1,\ldots,z_N$ be the complex points given by the stereographic projection $\pi_{\mathbb{S}^{2}}$ of $\omega_{N}$. 
Then, the polynomial $P(x) = \prod_{i=1}^N (x-z_i)$ satisfies $\mu_{\rm norm}(P)\leq N + o(N)$.
\end{cor}

Theorem~\ref{th:bez3} shows that if a set of $N$ points in the sphere has logarithmic potential very close to the minimum then the monic polynomial associated to the projection of these points has small condition number. 
Actually, this fact is the reason for the exact form of the problem posed by Shub and Smale that is nowadays known as Problem number 7 in Smale's list \cite{Smale}:
\begin{problem}[\cite{Smale}, Smale's 7th problem]\label{smale7}
Can one find $\omega_{N} = \{x_1,\ldots,x_N\}\subset\S$ such that $\mathcal{E}_{\log}(\omega_N)\leq \min\limits_{\mu_{N} \subset \S}\mathcal{E}_{\log} (\mu_{N}) +c\log N$ for some universal constant $c$?
\end{problem}
The minimal value of $\mathcal{E}_{\log}$ is not sufficiently well understood. 
Upper and lower bounds were given in \cite{BE18,10.2307/40234572,Wagner,BHS2012b,RSZ94,Dubickas,Brauchart2008}, and the last word is given in  \cite{BS18}, based on previous work \cite{SS12}.

%-------------------------------------------------------------------
%-------------------------------------------------------------------

\begin{thm}[\cite{BS18}*{Theorem~1.5}]\label{eq:as}
There exists $C_{\log}\neq 0$ independent of $N$ such that, as $N \longrightarrow \infty$,
\begin{equation*}
\min\limits_{\mu_{N} \subset \S}\mathcal{E}_{\log} (\mu_{N})=\kappa\,N^2-\frac12\,N\log N+C_{\log}\,N+o(N),
\end{equation*}
with $\mu_{N}$ a $N$-point subset of $\S$.
Moreover, 
\begin{equation*}
C_{\log} \leq -0.0556053\ldots
\end{equation*}
\end{thm}

%-------------------------------------------------------------------
%-------------------------------------------------------------------
In \cite{10.2307/40234572} the following lower bound is proved:
\begin{equation}\label{lb}
-0.2232823\ldots\leq
C_{\log}.
\end{equation}
%Combining  \cite{Dubickas} with \cite{BS18} it is known that
%\[
%-0.2232823526\ldots\leq C_{\log}\leq 2\log 2 
%+\frac12\log\frac23+3\log\frac{\sqrt\pi}{\Gamma(1/3)}=-0.0556053\ldots,
%\]
%and i
The upper bound for $C_{\log}$ has been conjectured to be an 
equality using two different approaches \cite{BHS2012b,BS18}.
The first constant on the asymptotic expansion is 
\begin{equation}\label{conjetura}
\kappa=\int_{\S}\int_{\S}\log |x-y |^{-1}\,d\sigma(x)d\sigma(y)=\frac{1}{2}-\log(2)<0
\end{equation}
the value of the continuous logarithmic energy.

No reverse result of Theorem \ref{th:bez3} is known until date, although in \cite{10.2307/23032778} authors prove a result that can be heuristically understood as an inverse.
Namely, they prove that polynomials with random coefficients, that are known to have small condition number on average, produce points with small logarithmic energy on average.
Here we prove a reverse statement for Theorem \ref{th:bez3} as a corollary of the following theorem.

\begin{thm}\label{th_cond_energy}
Let $(p_{N}(x))_{N\in \mathbb{N}}$ be a sequence of monic polynomials satisfying
\begin{equation*}
\mu_{norm}(p_{N}(x)) \leq B(N)
\end{equation*}
for a certain funtion $B$, let $\omega_{N} = \{ x_{1},\ldots ,x_N \} \subset \mathbb{S}^{2}$ be a set of points obtained through the inverse application of the projection $\pi_{\mathbb{S}^{2}}$ of the roots of the polynomials $p_{N}$.
Then 
\begin{equation*}
\mathcal{E}_{\log}(\omega_{N})
\leq
\kappa N ^2 -N\log\left( \frac{\sqrt{N(N+1)}}{2} \right) + N\log(B(N)).
\end{equation*}
\end{thm}

\begin{proof}
Since the condition number is bounded, for a fix $N$ $p_{N}$ has no roots with multiplicity higher than one and we can apply Lemma \ref{ene&cond} obtaining
\begin{multline*}
\mathcal{E}_{\log}(\omega_{N}) -
 \displaystyle\sum_{i = 1}^{N} \log (\mu_{\rm norm} (p_{N}, z_{i}))
\\
=
-N\log\left( \frac{\sqrt{N(N+1)}}{2} \right)
-N\log \left(\left(\int_{\S}\prod_{j=1}^N| p-x_j|^2d\sigma(p)\right)^{1/2} \right).
\end{multline*}
We use now Lemma \ref{LemmaR} to bound the second term on the right,
\begin{multline*}
\mathcal{E}_{\log}(\omega_{N}) -
 \displaystyle\sum_{i = 1}^{N} \log (\mu_{\rm norm} (p_{N}, z_{i}))
\leq
-N\log\left( \frac{\sqrt{N(N+1)}}{2} \right)
-N\log \left(e^{-\kappa N} \right)
\\
=
-N\log\left( \frac{\sqrt{N(N+1)}}{2} \right)
+\kappa N^2.
\end{multline*}
Finally we have
\begin{equation*}
\displaystyle\sum_{i = 1}^{N} \log (\mu_{\rm norm} (p_{N}, z_{i}))
\leq
N \log (B(N)).
\end{equation*}
and with this we conclude the proof.
\end{proof}

\begin{lem}\label{ene&cond}
Let $(p_{N}(x))_{N\in \mathbb{N}}$ be a sequence of monic polynomials defined by $p_{N} (x)= \prod_{i=1}^{N} (x-z_{i})$ with $z_{i}\neq z_{j}$ if $i\neq j$, let $\omega_{N} = \{ x_{1},\ldots ,x_N \}$ be a set of points in $\mathbb{S}^{2}$ obtained through the inverse application of the projection $\pi_{\mathbb{S}^{2}}$ of the roots of the polynomials $p_{N}(x)$.
Then 
\begin{multline*}
\mathcal{E}_{\log}(\omega_{N}) -
 \displaystyle\sum_{i = 1}^{N} \log (\mu_{\rm norm} (p_{N}, z_{i}))
\\
=
-N\log\left( \frac{\sqrt{N(N+1)}}{2} \right)
-N\log \left(\left(\int_{\S}\prod_{j=1}^N| p-x_j|^2d\sigma(p)\right)^{1/2} \right),
\end{multline*}
where $d\sigma$ is the sphere surface measure, normalized to satisfy $\sigma(\S)=1$.
%with $\lim\limits_{N\longrightarrow \infty}  \epsilon = 0$ .
\end{lem}

\begin{proof}

We just have to plug in the characterization of $\mu_{\rm norm} (p_{N}, z_{i})$ given in equation \eqref{def:workingmu1},
\begin{multline*}
\mathcal{E}_{\log}(\omega_{N}) -
\displaystyle\sum_{i = 1}^{N} \log (\mu_{\rm norm} (p_{N}, z_{i}))
\\
=
\mathcal{E}_{\log}(\omega_{N}) -
\displaystyle\sum_{i = 1}^{N} \log
\left(
\frac12\sqrt{N(N+1)} 
\frac{\left(\int_{\S}\prod_{j=1}^N| p-x_j|^2d\sigma(p)\right)^{1/2}}{\prod_{j\neq i}|x_i-x_j|}
\right)
%\\
%=
%\mathcal{E}_{\log}(\omega_{N}) -
%\displaystyle\sum_{i = 1}^{N}
%\left[
% \log
%\left(
%\frac{\sqrt{N(N+1)}}{2} 
%\right)
%\right.
%\\
%\left.
%+
%\log
%\left(
%\left(\int_{\S}\prod_{j=1}^N| p-x_j|^2d\sigma(p)\right)^{1/2}
%\right)
%-
%\log
%\left(
%\prod_{j\neq i}|x_i-x_j|
%\right)
%\right]
\\
=
\mathcal{E}_{\log}(\omega_{N}) 
-
N
 \log
\left(
\frac{\sqrt{N(N+1)}}{2} 
\right)
-
N\log \left(\left(\int_{\S}\prod_{j=1}^N| p-x_j|^2d\sigma(p)\right)^{1/2} \right)
\\
+
\displaystyle\sum_{i = 1}^{N}
\displaystyle\sum_{j\neq i}
\log
\left(
|x_i-x_j|
\right)
.
\end{multline*}
%We use now Lemma \ref{LemmaR} to bound the third term:
%\begin{multline*}
%%\mathcal{E}_{\log}(\omega_{N}) 
%%-
%%\displaystyle\sum_{i = 1}^{N}
%% \log
%%\left(
%%\frac{\sqrt{N(N+1)}}{2} 
%%\right)
%%-
%%\displaystyle\sum_{i = 1}^{N}
%%\log
%%\left(
%%\left(\int_{\S}\prod_{j=1}^N| p- x_j|^2d\sigma(p)\right)^{1/2}
%%\right)
%%\\
%%+
%%\displaystyle\sum_{i = 1}^{N}
%%\log
%%\left(
%%\prod_{j\neq i}|x_i-x_j|
%%\right)
%\mathcal{E}_{\log}(\omega_{N}) -
%\displaystyle\sum_{i = 1}^{N} \log (\mu_{\rm norm} (p_{N}, z_{i}))
%\\
%\leq
%\mathcal{E}_{\log}(\omega_{N}) 
%-
%N
% \log
%\left(
%\frac{\sqrt{N(N+1)}}{2} 
%\right)
%-
%\displaystyle\sum_{i = 1}^{N}
%\log
%e^{-\kappa N}
%+
%\displaystyle\sum_{i = 1}^{N}
%\displaystyle\sum_{j\neq i}
%\log
%\left(
%|x_i-x_j|
%\right)
%\\
%\leq
%\mathcal{E}_{\log}(\omega_{N}) 
%-
%N
% \log
%\left(
%\frac{\sqrt{N(N+1)}}{2} 
%\right)
%+
%\kappa N^2
%+
%\displaystyle\sum_{i = 1}^{N}
%\displaystyle\sum_{j\neq i}
%\log
%\left(
%|x_i-x_j|
%\right).
%\end{multline*}
The definition of the logarithmic energy is (see formula \eqref{defnener}):
\begin{equation*}
\mathcal{E}_{\log}(\omega_{N}) 
=
-
\displaystyle\sum_{i = 1}^{N}
\displaystyle\sum_{j\neq i}
\log
\left(
|x_i-x_j|
\right),
\end{equation*}
and so we conclude with
\begin{multline*}
\mathcal{E}_{\log}(\omega_{N}) -
 \displaystyle\sum_{i = 1}^{N} \log (\mu_{\rm norm} (p_{N}, z_{i}))
\\
=
-N\log\left( \frac{\sqrt{N(N+1)}}{2} \right)
-N\log \left(\left(\int_{\S}\prod_{j=1}^N| p-x_j|^2d\sigma(p)\right)^{1/2} \right).
\end{multline*}
%\begin{multline*}
%\mathcal{E}_{\log}(\omega_{N}) -
%\displaystyle\sum_{i = 1}^{N} \log (\mu_{\rm norm} (p_{N}, z_{i}))
%\leq
%\kappa N^2
%-N\log\left( \frac{\sqrt{N(N+1)}}{2} \right)
%\\
%<
%\kappa N^2
%-N\log(N)
%+\log(2)N%(\log(2)+\epsilon)
%.
%\end{multline*}
%with $\epsilon$ such that $\lim\limits_{N\longrightarrow \infty}  \epsilon = 0$.
\end{proof}

Combining Lemma \ref{ene&cond}, Lemma \ref{LemmaR}  and Theorem \ref{eq:as} we obtain the following corollary.
\begin{cor}\label{nuevoBezIII}
Let $(p_{N}(x))_{N\in \mathbb{N}}$ be a sequence of monic polynomials defined by $p_{N} (x)= \prod_{i=1}^{N} (x-z_{i})$, then 
\begin{equation*}
\displaystyle\sum_{i = 1}^{N} \log (\mu_{\rm norm} (p_{N}, z_{i}))
\geq
\frac{N\log(N)}{2}
+(C_{\log} - \log(2))N
+ o(N).
\end{equation*}
\end{cor}

\subsection{Proof of Theorem \ref{thmain1}}
We only have to take $B(N) =C\sqrt{N}$ in Theorem \ref{th_cond_energy} and we obtain
\begin{equation}\label{miecuacion}
\mathcal{E}_{\log}(\omega_{N})
\leq
\kappa N ^2 - \frac{N\log(N)}{2} + \log(2C) N - \frac{N}{2}\log\left( 1+ \frac{1}{N} \right).
\qedhere
\end{equation}
\begin{flushright}
$\qed$
\end{flushright}

\begin{obs}
Note that the asymptotic provided by Theorem \ref{thmain1} is really close to the minimal value of the logarithmic energy (the current knowledge of such quantity is given in Theorem \ref{eq:as}).
%In \cite{BEMOC} it was proved that giving explicitelly a family of polynomials with optimal condition number was easier than giving a family of spherical points with logarithimic energy close to the minimal in the sense of Problem \ref{smale7}.
Theorem \ref{thmain1} may reverse the approach of Problem \ref{smale7}, focusing on building families of points with optimal condition number as proposed in \cite{BEMOC}.
\end{obs}

\begin{obs}
If we  repeat the proof of \cite[Lemma 1.1]{BEMOC}, extending the computations to the $N$ term we have that
\begin{equation}\label{eso}
\mu_{\rm norm}(p_{N})\geq \frac{e^{C_{\log}}}{2} \sqrt{N} + o(\sqrt{N})
\end{equation}
for any sequence of monic polynomials $(p_{N})_{N}$. %, in particular, for the one we are working with.
Let us take a sequence of monic polynomials $(p_{N})_{N}$ with optimal condition number, i.e. $\mu_{norm}(p_{N}(x))\leq C\sqrt{N}$ for some universal constant $C$.
Combining equation \eqref{eso} with the bound for $C_{\log}$ given in equation \eqref{lb} we obtain
\begin{equation*}
0.4
\leq
\frac{e^{C_{\log}}}{2}
\leq 
C.
\end{equation*}
If %there exists a sequence of monic polynomials $(p_{N})_{N}$ with optimal condition number, i.e. $\mu_{norm}(p_{N}(x))\leq C\sqrt{N}$ and 
$C = \frac{e^{C_{\log}}}{2}$ then the third term on the asymptotic expansion \eqref{miecuacion} is equal to the third term on the minimal logarithmic energy asymptotic expansion.
\end{obs}
%------------------------------------------------------------
%------------------------------------------------------------
%------------------------------------------------------------
%------------------------------------------------------------
\section{Proof of Theorem \ref{thmain2}}\label{sec4}

The main tool that we use in this section is a theorem that we  deduced from a formula presented in \cite{10.2307/23032778} and proved in \cite{Beltran:2015:FLF:2729553.2729572} relating the logarithmic energy of a set of points on the Riemann sphere $\mathbb{S}$ with a quantity  involving the condition number of a polynomial and a quotient involving norms of polynomials.
%Let $\hat{\omega}_{N} = \{ \hat{z}_{1},\ldots,\hat{z}_{N} \}$ be a set of spherical points in $\mathbb{S}$, let $z_{i} = \pi_{\mathbb{S}} (\hat{z}_{i})$ be complex points and let $p_{N} = \prod_{i=1}^{N} (x-z_{i})$ for all $N\in \mathbb{N}$. 
%Let $\mathcal{E}_{\log}^{\mathbb{S}}$ define the logarithmic energy on the Riemann sphere $\mathbb{S}$.
%Then, the formula, as presented in \cite{10.2307/23032778} reads
%\begin{equation}\label{Lemita}
%\mathcal{E}_{\log}^{\mathbb{S}}(\hat{\omega}_{N}) = 
%\frac{1}{2} \underbrace{\displaystyle\sum_{i = 1}^{N} \log (\mu_{\rm norm} (p_{N}, z_{i}))}_{A} 
%+ 
%\frac{N}{2} \underbrace{\log \left( \frac{\displaystyle \prod_{i=1}^{N} \sqrt{1+|z_{i}|^{2}}}{|| p_{N} ||} \right)}_{B} - \frac{N}{4} \log(N).
%\end{equation}

%Here, we characterize this quotient without any reference to the logarithmic energy or the condition number.

\begin{lem}[]\label{ABS}
Let $\omega_{N} = \{x_{1},\ldots,x_{N} \}$ be a set of different spherical points in $\mathbb{S}^{2}$, $z_{i} = \pi_{\mathbb{S}^{2}} (x_{i})$  be complex points and $p_{N}(x) = \prod_{i=1}^{N} (x-z_{i})$ for all $N\in \mathbb{N}$. 
Then, 
\begin{multline*}
\mathcal{E}_{\log}(\omega_{N}) 
= 
 \displaystyle\sum_{i = 1}^{N} \log (\mu_{\rm norm} (p_{N}, z_{i}))
+ 
N \log \left( \frac{\displaystyle \prod_{i=1}^{N} \sqrt{1+|z_{i}|^{2}}}{|| p_{N} ||} \right) 
\\
-\log(2)N^2
- \frac{N\log(N)}{2} 
+\log(2)N.
\end{multline*}
\end{lem}

\begin{proof}
%The formula is not proved in paper \cite{10.2307/23032778}, nevertheless, a proof for the formula can be found in \cite{Beltran:2015:FLF:2729553.2729572}.
We redirect the reader to the proof presented in \cite{Beltran:2015:FLF:2729553.2729572}*{Lemma~1.6} with a few comments.
%There are slightly variations in between the elements used in \cite{Beltran:2015:FLF:2729553.2729572} and the ones that we use here.
In \cite{Beltran:2015:FLF:2729553.2729572} authors take as the logarithmic energy a half of our quantity for the logarithmic energy. 
Also, the statement is a bit ambiguous and it may seem that the formula only works for minimizers of the logarithmic energy.
Nevertheless, a careful reading of the proof of \cite{Beltran:2015:FLF:2729553.2729572}*{Lemma~1.6} allows the reader to deduce that equation is valid for any set of different points in the sphere.
Through Lemma \ref{lem_ener} we can translate the formula from \cite{Beltran:2015:FLF:2729553.2729572} into the notation that we are following here, with the appropiate logarithmic energy on the unit sphere $\mathbb{S}^{2}$.
\end{proof}

From Lemma \ref{ABS} we deduce a characterization for the quotient of the product of the norms of a group of monomials and the norm of the product of these monomials, with the Bombieri-Weyl norm.

\begin{thm}\label{miformula}
Given a set of complex points $z_{1},\ldots,z_{N}$, let $z_{i} = \pi_{\S}(x_{i})$ for $1\leq i \leq N$.
Then we have
\begin{equation*}
\frac{\displaystyle \prod_{i=1}^{N} ||x-z_{i}||}{\left|\left| \displaystyle\prod_{i=1}^{N} (x-z_{i}) \right|\right|}
=
\frac{2^{N}}{\sqrt{N+1}}\left(\int_{\S}\prod_{i=1}^N| p-x_i|^2d\sigma(p)\right)^{-1/2}
\end{equation*}
where $||*||$ is the Bombieri-Weyl norm.
\end{thm}

\begin{proof}
This theorem is a consequence of Lemma \ref{ABS} and Lemma \ref{ene&cond}.
Note that %the second term on the right is just 
\begin{equation*}
\frac{\displaystyle \prod_{i=1}^{N} \sqrt{1+|z_{i}|^{2}}}{|| p_{N} ||} 
=
\frac{\displaystyle \prod_{i=1}^{N} ||x-z_{i}||}{\left|\left| \displaystyle\prod_{i=1}^{N} (x-z_{i}) \right|\right|}.
\end{equation*}
Let us suppose that $z_{i} \neq z_{j}$ if $i\neq j$, combining both results we obtain
\begin{multline*}
N \log \left( \frac{\displaystyle \prod_{i=1}^{N} ||x-z_{i}||}{\left|\left| \displaystyle\prod_{i=1}^{N} (x-z_{i}) \right|\right|} \right) 
-\log(2)N^2
- \frac{N\log(N)}{2} 
+\log(2)N
\\
=
-N\log\left( \frac{\sqrt{N(N+1)}}{2} \right)
-N\log \left(\left(\int_{\S}\prod_{i=1}^N| p-x_i|^2d\sigma(p)\right)^{1/2} \right).
\end{multline*}
A few manipulations on the terms above lead us to
\begin{multline*}
N \log \left( \frac{\displaystyle \prod_{i=1}^{N} ||x-z_{i}||}{\left|\left| \displaystyle\prod_{i=1}^{N} (x-z_{i}) \right|\right|} \right) 
=
\log(2)N^2
- \frac{N\log(N)}{2} 
- \frac{N\log \left( 1 + \frac{1}{N} \right)}{2} 
\\
-N\log \left(\left(\int_{\S}\prod_{i=1}^N| p-x_i|^2d\sigma(p)\right)^{1/2} \right)
\end{multline*}
and we conclude
\begin{equation*}
\frac{\displaystyle \prod_{i=1}^{N} ||x-z_{i}||}{\left|\left| \displaystyle\prod_{i=1}^{N} (x-z_{i}) \right|\right|}
=
\frac{2^{N}}{\sqrt{N+1}}\left(\int_{\S}\prod_{i=1}^N| p-x_i|^2d\sigma(p)\right)^{-1/2}.
\end{equation*}
Let us consider the case where $z_{i} = z_{j}$ for $i \neq j$.
We can take limits in the previous expression
\begin{multline*}
\lim\limits_{z_{i}\rightarrow z_{j}}
\left(\int_{\S} | p-x_1|^2 | \ldots | p-x_i|^2 \ldots | p-x_j|^2 \ldots | p-x_N|^2 d\sigma(p)\right)^{-1/2}
\\
=
\lim\limits_{z_{i}\rightarrow z_{j}}
\left(\int_{\S} | p-x_1|^2 | \ldots | p-x_j|^4 \ldots | p-x_N|^2 d\sigma(p)\right)^{-1/2}
<
\infty
\end{multline*}
and by continuity of the function, we can extend the characterization to polynomials with roots with multiplicity higher than one.
%Note that Lemma \ref{ene&cond} is stated for a set of different points $z_{1},\ldots,z_{N} \in \mathbb{C}$, but we can conclude Theorem \ref{miformula} by continuity on the limit case $z_{i} = z_{j}$ for $i\neq j$.
\end{proof}

Although the quantity on the right of the equation in Theorem \ref{miformula} is difficult to compute, we can easily bound it, as we do in Theorem \ref{norms}.%, obtaining Theorem \ref{thmain2}.

\begin{thm}\label{norms}
Given a set of complex points $z_{1},\ldots,z_{N}$,  we have
\begin{equation*}
\displaystyle \prod_{i=1}^{N} ||x-z_{i}||
\leq
%\frac{1}{4e^{\epsilon}}
\sqrt{\frac{e^{N}}{N+1}}
\left|\left| \prod_{i=1}^{N} (x-z_{i}) \right|\right|,
\end{equation*}
where $||*||$ is the Bombieri-Weyl norm.% and $\lim\limits_{N\longrightarrow\infty} \epsilon = 0$.
\end{thm}

\begin{proof}
We just have to bound the expression in Theorem \ref{miformula} using Lemma \ref{LemmaR}:
\begin{equation*}
\frac{\displaystyle \prod_{i=1}^{N} ||x-z_{i}||}{\left|\left| \displaystyle\prod_{i=1}^{N} (x-z_{i}) \right|\right|}
=
\frac{2^{N}}{\sqrt{N+1}}\left(\int_{\S}\prod_{j=1}^N| p-x_j|^2d\sigma(p)\right)^{-1/2}
\leq
\frac{2^{N}e^{\kappa N}}{\sqrt{N+1}}.
\end{equation*}
Now we rewrite $2^{N} = e^{N\log (2)}$ and so,
\begin{equation*}
2^{N}e^{\kappa N} = e^{\frac{N}{2}}.
\end{equation*}

\end{proof}

%------------------------------------------------------------
%------------------------------------------------------------
%------------------------------------------------------------
%------------------------------------------------------------

\subsection{Sharpness of Theorem \ref{thmain2}}\label{sec5}

We have at least two natural candidates when searching for polynomials that make inequality from Theorem \ref{norms} sharp: polynomials with optimal condition number and polynomials coming from minimizers of the logarithmic energy.

%------------------------------------------------------------
%------------------------------------------------------------
%------------------------------------------------------------
%------------------------------------------------------------

\subsubsection{Polynomials with minimal condition number}

%First we present a lemma that can be understood as a sharp result for Lemma \ref{ene&cond}.

\begin{lem}\label{sharpcondi}
Let $(p_{N}(x))_{N\in \mathbb{N}}$ be a sequence of polynomials defined by $p_{N}(x) = \prod_{i=1}^{N} (x-z_{i})$ with optimal condition number, i.e. $\mu_{norm}(p_{N}(x)) \leq C\sqrt{N}$ for some universal constant $C$, 
let $\omega_{N} = \{ x_{1},\ldots ,x_N \} \subset \mathbb{S}^{2}$ be a set of points obtained by the inverse application of the projection $\pi_{\mathbb{S}^{2}}$ of the roots of the polynomials.
Then 
\begin{equation*}
\mathcal{E}_{\log}(\omega_{N}) -
\displaystyle\sum_{i = 1}^{N} \log (\mu_{\rm norm} (p_{N}, z_{i}))
\geq
\kappa N^2
-
N \log(N)
 + (C_{\log} - \log(C))N
 +o(N)
.
\end{equation*}
\end{lem}

\begin{proof}
We have $\mu_{norm}(p_{N}(x)) \leq C\sqrt{N}$ for some universal constant $C$, then
\begin{equation*}
%\mu_{\rm norm} (p_{N}, z_{i}) \leq \mu_{norm}(p_{N}(x))\leq C\sqrt{N}
%\\
%\Rightarrow\ref{ene&cond} and Theorem \ref{norms}
%\log \left(\mu_{\rm norm} (p_{N}, z_{i})\right) 
%\leq 
%\log \left( C\sqrt{N}\right)
%=
%\frac{\log(N)}{2} + \log(C)
%\\
%\Rightarrow
-\displaystyle\sum_{i = 1}^{N}
\log \left(\mu_{\rm norm} (p_{N}, z_{i})\right) 
\geq 
%-\displaystyle\sum_{i = 1}^{N}
%\left( \frac{\log(N)}{2} + \log(C) \right)
%=
-\frac{N\log (N)}{2} -\log(C)N.
\end{equation*}
From Theorem \ref{eq:as} we have that 
\begin{equation*}
\mathcal{E}_{\log}(\omega_{N})
\geq
\min\limits_{\mu_{N} \subset \S}\mathcal{E}_{\log} (\mu_{N})=
\kappa\,N^2-\frac12\,N\log N+C_{\log}\,N+o(N),
\end{equation*}
so we have that 
\begin{multline*}
\mathcal{E}_{\log}(\omega_{N}) -
\displaystyle\sum_{i = 1}^{N} \log (\mu_{\rm norm} (p_{N}, z_{i}))
\\
\geq
\kappa N^{2} - \frac{N\log N}{2} + C_{\log}N + o(N) -\frac{N\log (N)}{2} -\log(C)N
\end{multline*}
and the proof is finished.
\end{proof}

%------------------------------------------------------------
%------------------------------------------------------------
%------------------------------------------------------------
%------------------------------------------------------------

%Here we look for a set of points on the sphere satisfying

\begin{thm}\label{sharp2}
Let $(p_{N}(x))_{N\in \mathbb{N}}$ be a sequence of polynomials defined by $p_{N}(x) = \prod_{i=1}^{N} (x-z_{i})$ with optimal condition number, i.e. $\mu_{norm}(p_{N}(x)) \leq C\sqrt{N}$ for some universal constant $C$, then
\begin{equation*}
\displaystyle \prod_{i=1}^{N} ||x-z_{i}||
\geq
\frac{
e^{C_{\log} + o(1)}
}{2C}
\sqrt{\frac{e^{N}}{N}}
\left|\left| \prod_{i=1}^{N} (x-z_{i}) \right|\right|,
\end{equation*}
where $||*||$ is the Bombieri-Weyl norm.
\end{thm}

\begin{proof}
%It's just a consequence of Lemma \ref{sharpcondi}.
From Lemma \ref{ABS} and Lemma \ref{sharpcondi} we have that
\begin{multline*}
N\log \left( \frac{\displaystyle \prod_{i=1}^{N} ||x-z_{i}||}{\left|\left| \displaystyle\prod_{i=1}^{N} (x-z_{i}) \right|\right|} \right) 
- \log(2)N^2
- \frac{N}{2} \log(N)
+ \log(2)N
\\
\geq 
\kappa N^2
-N \log(N)
 + (C_{\log} - \log(C))N
 +o(N)
\\
\Rightarrow
\log \left( \frac{\displaystyle \prod_{i=1}^{N} ||x-z_{i}||}{\left|\left| \displaystyle\prod_{i=1}^{N} (x-z_{i}) \right|\right|} \right) 
\geq
\frac{N}{2}
-
\frac{\log(N)}{2}
+ (C_{\log} - \log(C) - \log(2))
+o(1).
\end{multline*}
Some algebraic manipulations lead us to
\begin{equation*}
\displaystyle \prod_{i=1}^{N} ||x-z_{i}||
\geq
e^{C_{\log} +o(1) - \log(2C)}
\sqrt{\frac{e^{N}}{N}}
\left|\left| \displaystyle\prod_{i=1}^{N} (x-z_{i}) \right|\right|.
\end{equation*}
\end{proof}

%------------------------------------------------------------
%------------------------------------------------------------
%------------------------------------------------------------
%------------------------------------------------------------

\subsubsection{Minimizers of the logarithmic energy}

The same computations can be done for a set of minimizers of the logarihmic energy $\omega_{N} = \{ x_{1},\ldots ,x_N \}\subset\mathbb{S}^{2}$ and its associated complex points $z_1,\ldots,z_N$.
%In this case, the bound for the condition number is worse than in the previous case, and so 
Note that in this case, the best bound known for the condition number is given in Corollary \ref{corBeIII}. % and it is not optimal, hence we obtain a bound of different order.
Namely,
%
%\begin{thm}\label{sharpener}
let $\omega_{N} = \{ x_{1},\ldots ,x_N \}\subset\mathbb{S}^{2}$ be a set of minimizers of the logarithmic energy.
Let $z_1,\ldots,z_N$ be the set of complex points obtained through the projection $\pi_{\mathbb{S}^{2}}$ of $\omega_{N}$ and let $(p_{N}(x))_{N\in \mathbb{N}}$ be a sequence of polynomials defined by $p_{N}(x) = \prod_{i=1}^{N} (x-z_{i})$.
Then 
\begin{equation*}
\mathcal{E}_{\log}(\omega_{N}) -
\displaystyle\sum_{i = 1}^{N} \log (\mu_{\rm norm} (p_{N}, z_{i}))
\geq
\kappa\,N^2-\frac32\,N\log N
+
o(N\log(N))
\end{equation*}
%\end{thm}
%
%\begin{proof}
%From Theorem \ref{eq:as} we know that
%\begin{equation*}
%\min\limits_{\mu_{N} \subset \S}\mathcal{E}_{\log} (\mu_{N})
%\geq
%\kappa\,N^2-\frac12\,N\log N+C_{\log}\,N+o(N),
%\end{equation*}
%and from Theorem \ref{th:bez3} we have that
%\begin{equation*}
%\mu_{\rm norm}(p_{N})\leq \sqrt{N(N+1)}.
%\end{equation*}
%Joining this results we have
%\begin{multline*}
%\mathcal{E}_{\log}(\omega_{N}) -
%\displaystyle\sum_{i = 1}^{N} \log (\mu_{\rm norm} (f, z_{i}))
%\\
%\geq
%\kappa\,N^2-\frac12\,N\log N+C_{\log}\,N+o(N)
%-
%\displaystyle\sum_{i = 1}^{N} \log 
%\left(
%\sqrt{N(N+1)}
%\right)
%\\
%=
%\kappa\,N^2-\frac12\,N\log N+C_{\log}\,N+o(N)
%-
%N\log(N)
%+
%o(N).
%\\
%=
%\kappa\,N^2-\frac32\,N\log N
%+
%o(N).
%\end{multline*}
%
%\end{proof}
and
%------------------------------------------------------------
%------------------------------------------------------------
%------------------------------------------------------------
%------------------------------------------------------------

%\begin{thm}\label{sharp2ener}
%Let $\omega_{N} = \{ x_{1},\ldots ,x_N \}\subset\mathbb{S}^{2}$ be a set of minimizers of the logarithmic energy and let $z_1,\ldots,z_N$ be the set of complex points obtained applying the map $\pi_{\mathbb{S}^{2}}$ to the set $\omega_{N}$, then
\begin{equation*}
\displaystyle \prod_{i=1}^{N} ||x-z_{i}||
\geq
o(1)
\frac{e^{
\frac{N}{2}
}}{N}
\left|\left| \prod_{i=1}^{N} (x-z_{i}) \right|\right|,
\end{equation*}
where $||*||$ is the Bombieri-Weyl norm.
%\end{thm}

%\begin{proof}
%It's just a consequence of Theorem \ref{sharpcondi} and Theorem \ref{norms}.
%Again using Lemma \ref{ABS} we have that
%\begin{multline*}
%N\log \left( \frac{\displaystyle \prod_{i=1}^{N} ||x-z_{i}||}{\left|\left| \displaystyle\prod_{i=1}^{N} (x-z_{i}) \right|\right|} \right) 
%- \log(2)N^2
%- \frac{N}{2} \log(N)
%+ \log(2)N
%\\
%\geq 
%\kappa\,N^2-\frac32\,N\log N
%+
%o(N)
%\\
%\Rightarrow
%\log \left( \frac{\displaystyle \prod_{i=1}^{N} ||x-z_{i}||}{\left|\left| \displaystyle\prod_{i=1}^{N} (x-z_{i}) \right|\right|} \right) 
%\geq
%\frac{N}{2}
%-
%\log(N)
%-\log(2)
%+o(1).
%\end{multline*}
%and following the scheme of the proof of Theorem \ref{norms} we end up with
%\begin{equation*}
%\displaystyle \prod_{i=1}^{N} \sqrt{1+|z_{i}|^{2}}
%\geq
%\frac{
%e^{\frac{N}{2}}
%}{2N
%}
%|| p_{N} ||.
%\end{equation*}
%\end{proof}

%\subsection{Roots with multiplicity higher than 1}

As we have shown in Theorem \ref{sharp2}, polynomials with optimal condition number, i.e. whose roots, taken as points on the proyective space, are well separated on the Riemann sphere, attain the greatest difference between the norm of the product and the product of the norms.
The oposite behaviour happens when we take all roots to be equal, the same point on the Riemann sphere.

%Note that the statement of Theorem \ref{ene&cond} allows us to consider polynomials $p_{N} (x)$ whose roots have multiplicity higher than one.
%As we mention in Section \ref{sec2}, the condition number of a polynomial on a root with multiplicity higher than one is $+\infty$, but if we look carefully to the expression above, we can see how this term cancels with the logarithmic energy, and so, Theorem \ref{ene&cond} doesn't ask all the roots to be different.
%We will give some more details abouts roots with high multiplicity in Section \ref{sec6}.

\begin{lem}\label{lemaminimo}
Let $a\in\mathbb{C}$, then
\begin{equation*}
 ||x-a||^{N}
=
\left|\left| (x-a)^{N} \right|\right|.
\end{equation*}
\end{lem}

\begin{proof}
It is an easy computation:
\begin{multline*}
\left|\left| (x-a)^{N} \right|\right|
=
\left|\left| \sum_{i=0}^{N} {N \choose i} x^{N-i} (-a)^{i} \right|\right|
=
\sqrt{\sum_{i=0}^{N} {N \choose i}^{-1}\left| {N \choose i}(-a)^{i}  \right|^{2} }
\\
=
\sqrt{\sum_{i=0}^{N} {N \choose i} |a|^{2i} }
=
(1 + |a|^2)^{\frac{N}{2}}
=
 ||x-a||^{N}.
\end{multline*}
\end{proof}

%------------------------------------------------------------
%------------------------------------------------------------
%------------------------------------------------------------
%------------------------------------------------------------

%------------------------------------------------------------
%------------------------------------------------------------
%------------------------------------------------------------
%------------------------------------------------------------
\section{Open problems}\label{sec6}

\subsection{On Problem \ref{mainprob}}
Theorems \ref{norms} and \ref{sharp2} pose the following scenario:
\begin{equation*}
\displaystyle \prod_{i=1}^{N} ||x-z_{i}||
\leq
K_{N}
\sqrt{\frac{e^{N}}{N+1}}
\left|\left| \prod_{i=1}^{N} (x-z_{i}) \right|\right|,
\end{equation*}
with $K_{N}$ the minimal value satisfying the previous inequality, i.e. the minimal value such that for any set of $N$ complex numbers $z_1,\ldots,z_N$ the inequality is satisfied.
We know that
\begin{equation*}
0
<
K_{N}
\leq
1.
\end{equation*}
%where $-0.2232823\ldots\leq C_{\log} \leq -0.0556053\ldots$ (see Theorem \ref{eq:as}) and numerical experiments in \cite{BEMOC} suggest that there exists $C\leq 15$.
Actually, we can compute $K_{N}$ for the first values of $N$.

For $N=2$, $K_{2}\frac{e}{\sqrt{3}} = \sqrt{2}$, as given in the Bombieri inequality (Theorem \ref{thmBombi}), so we have 
\begin{equation*}
K_{2} = \frac{\sqrt{6}}{e}\approx 0.9011\ldots
\end{equation*}
For $N=3$, we can take $z_{1},z_{2},z_{3}$ as the three roots of $x^3 =1$, obtaining
\begin{equation*}
K_{3} = \frac{4}{e\sqrt{e}} \approx 0.8925\ldots
\end{equation*}
Finally, for $N= 4$ we have to take the complex points coming from the stereographic proyection of any regular tetrahedron inscribed in $\S$.
This give us
\begin{equation*}
K_{4} = \frac{3\sqrt{5}}{e^2} \approx 0.9078\ldots
\end{equation*}
These simple computations lead us to the following question.
\begin{prob}\label{myprob}
Is there a universal constant $A$ such that $K_{N} = A + o(1)$ for all $N\in\mathbb{N}$?
\end{prob}

Now we quote \cite{beltran_2012}:
\begin{displayquote}
\textit{Experiments suggest that minimising $\mathcal{E}_{\log}$ is a problem similar to minimising the sum of $\log \mu_{\rm norm}(p_{N}, z_i)$, and to maximising the quotient $\frac{\displaystyle \prod_{i=1}^{N} ||x-z_{i}||}{\left|\left| \displaystyle\prod_{i=1}^{N} (x-z_{i}) \right|\right|}.$}
\end{displayquote}
This statement suggests that if the answer to Problem \ref{myprob} is positive that may help with the resolution of Problem \ref{smale7}.

%------------------------------------------------------------

%\noindent \textbf{Theorem \ref{sharp2ener}}
%\textit{Let $\omega_{N} = \{ x_{1},\ldots ,x_N \}$ be a set of minimizers of the logarithmic energy and let $z_1,\ldots,z_N$ be the set of complex point obtained by the inverse the map $\pi_{\mathbb{S}^{2}}$, then
%\begin{equation*}
%\displaystyle \prod_{i=1}^{N} ||x-z_{i}||
%\gtrapprox
%\frac{e^{
%\frac{N}{2}
%}}{N}
%\left|\left| \prod_{i=1}^{N} (x-z_{i}) \right|\right|,
%\end{equation*}
%where $||*||$ is the Bombieri-Weyl norm.}

%------------------------------------------------------------
%------------------------------------------------------------
\subsection{On Problem \ref{prob1}}
Theorem \ref{thmain2} provides us with a new bound for Problem \ref{eq:main}.
Let $P_{i}(x) = \prod_{j=1}^{N_{i}} (x-z_{j}^{(i)})$ be polynomials of degree $N_{i}$ for $1\leq i \leq m$, then using Lemma \ref{lemlog} and Theorem \ref{thmain2} we have that
\begin{multline*}
||P_{1}||\ldots||P_{m}||
=
\left| \left|
\prod_{j=1}^{N_{1}} (x-z_{j}^{(1)})
\right|\right|
\ldots
\left|\left|
\prod_{j=1}^{N_{m}} (x-z_{j}^{(m)})
\right|\right|
\\
\leq
\prod_{j=1}^{N_{1}} 
\left| \left|
(x-z_{j}^{(1)})
\right|\right|
\ldots
\prod_{j=1}^{N_{m}}
\left|\left|
 (x-z_{j}^{(m)})
\right|\right|
\leq
\sqrt{\frac{e^{(N_{1}+\ldots+N_{m})}}{(N_{1}+\ldots+N_{m}+1)}}
||P_{1}\ldots P_{m}||.
\end{multline*}
This new bound is better than the one provided by Corollary \ref{corBombi} only in certain cases.
%For example, if we have $m=2$ and $N_{1} = N_{2}$, then our bound reads
%\begin{equation*}
%\sqrt{\frac{e^{2N_{1}}}{2N_{1}}}
%\end{equation*}
%and the bound given by Corollary \ref{corBombi} reads
%\begin{equation*}
%\sqrt{\frac{(2N_{1})!}{2(N_{1}!)}}
%\approx
%2^{N_{1}}.
%\end{equation*}
For example, if we take $P_{i}(x) = (x-z^{(i)})$, then Theorem \ref{thmain2} gives us the bound
\begin{equation*}
||P_{1}||\ldots||P_{m}||
\leq
\sqrt{\frac{e^{m}}{m+1}}
||P_{1}\ldots P_{m}||,
\end{equation*}
which is clearly smaller than the bound provided by Corollary \ref{corBombi}:
\begin{equation*}
||P_{1}||\ldots||P_{m}||
\leq
\sqrt{m!}
||P_{1}\ldots P_{m}||.
\end{equation*}
If instead we take $P_{1}(x) = \prod_{j=1}^{N-1} (x-z_{j})$, $P_{2}(x) = (x-a)$ then the bound given by Theorem \ref{thmain2} reads
\begin{equation*}
||P_{1}||||P_{2}||
\leq
\sqrt{\frac{e^{N}}{N+1}}
||P_{1} P_{2}||,
\end{equation*}
and the one given by Corollary \ref{corBombi}:
\begin{equation*}
||P_{1}||||P_{2}||
\leq
\sqrt{N}
||P_{1} P_{2}||.
\end{equation*}
The new inequality for products of polynomials may read as in the following theorem.
\begin{thm}
Let $P_{1},\ldots,P_{m}$ be univariate polynomials of degrees $k_{i}$ respectively.
Then
\begin{equation*}
||P_{1}||\ldots||P_{m}||
\leq
\min
\left\lbrace
\sqrt{\frac{(k_{1}+\ldots +k_{m})!}{k_{1}!\ldots k_{m}!}}
,
%\frac{1}{4}
\sqrt{\frac{e^{(k_{1}+\ldots+k_{m})}}{(k_{1}+\ldots+k_{m}+1)}}
\right\rbrace
||P_{1}\ldots P_{m}||
.
\end{equation*}
\end{thm}
%------------------------------------------------------------
%------------------------------------------------------------
%------------------------------------------------------------
%------------------------------------------------------------

\section{Auxiliary lemmas}\label{sec7}

In this section we present some of the lemmas used for proving the previous statements.

\begin{lem}\label{LemmaR}
Let $\kappa$ be as defined in equation \eqref{conjetura} and let $\omega_{N} = \{ x_{1},\ldots ,x_N \}\subset\mathbb{S}^{2}$, then
\begin{equation*}
\left(\int_{ \S}\prod_{i=1}^N|p-x_i|^2\,d\sigma(p)\right)^{\frac{1}{2}}
\geq e^{-\kappa N}.
\end{equation*}
\end{lem}

\begin{proof}
Using Jensen's inequality we have
\begin{multline*}
\log \left(\int_{ \S}\prod_{i=1}^N|p-x_i|^2\,d\sigma(p)\right)^{\frac{1}{2}}
=
\frac{1}{2}\log \int_{ \S}\prod_{i=1}^N|p-x_i|^2\,d\sigma(p)
\\
\geq 
\frac{1}{2} \int_{  \S}\log\prod_{i=1}^N|p-x_i|^2\,d\sigma(p)
=
\sum_{i=1}^N \int_{  \S}\log|p-x_i|\,d\sigma(p)=-\kappa N,
\end{multline*}
and hence $\left(\int_{ \S}\prod_{i=1}^N|p-x_i|^2\,d\sigma(p)\right)^{\frac{1}{2}}\geq e^{-\kappa N}$. 
\end{proof}

\begin{lem}\label{lemlog}
Let $P,Q$ be two univariate polynomials with complex coefficients.
Then,
\begin{equation*}
|| P||  ||Q ||
\geq 
|| PQ ||,
\end{equation*}
where $||*||$ is the Bombieri-Weyl norm.
\end{lem}

\begin{proof}
That follows from the fact that Bombieri norm is an algebra norm.
In Lemma \ref{lemaminimo} we give a polynomial satisfying the equation.
%In \cite{Reznick1993AnIF} the author proves that $||pq||=||p||⋅||q||$ if and only if, after a unitary change of variables, $p$ and $q$ are forms in disjoint sets of variables.
\end{proof}

\begin{lem}\label{lemcond}
For any polynomial $P$ and any of its roots $z$,
\begin{equation*}
\mu_{\rm norm}(P,z) \geq 1.
\end{equation*}
\end{lem}

\begin{proof}
See \cite{SHUB19934}*{pg. 6}.
\end{proof}

\begin{lem}\label{lem_ener}
Let $\hat{\omega}_{N} = \{ \hat{z}_{1},\ldots,\hat{z}_{N} \}$ be a set of points on the Riemann sphere $\mathbb{S}$ and let $\omega_{N} = \{ x_{1},\ldots,x_{N} \}\subset\mathbb{S}^{2}$ be the proyection of the set $\hat{\omega}_{N}$ by the map $h$.
Then,
\begin{equation*}
\mathcal{E}_{\log}^{\mathbb{S}}(\hat{\omega}_{N}) 
=
\mathcal{E}_{\log}(\omega_{N})
- 
\log(2)(N^2 - N).
\end{equation*}
%We usually denote by $\mathcal{E}_{\log}^{\mathbb{S}^{2}} = \mathcal{E}_{\log}$
\end{lem}

\begin{proof}
Using the homotethic transformation $h$,
\begin{multline*}
\mathcal{E}_{\log}(\omega_N)
=
-
\sum_{i\neq j}\log ||x_i-x_j||
=
-
\sum_{i\neq j}\log ||2\hat{z}_{i} - (0,0,1)-(2\hat{z}_{j} - (0,0,1))||
\\
=
-
\sum_{i\neq j}\log ||2\hat{z}_{i} -2\hat{z}_{j}||
=
-
\sum_{i\neq j}\log 2
-
\sum_{i\neq j}\log ||\hat{z}_{i}-\hat{z}_{j}||
\\
=
-
\log (2)N(N-1)
-
\sum_{i\neq j}\log ||\hat{z}_{i}-\hat{z}_{j}||
\\
=
-\log(2)(N^2 - N) +
\mathcal{E}_{\log}^{\mathbb{S}}(\omega_N).
\end{multline*}
\end{proof}

%------------------------------------------------------------
%------------------------------------------------------------
%------------------------------------------------------------
%------------------------------------------------------------

\subsection*{Acknowledgements}
When I was a PhD student my advisor, Carlos Beltrán, pointed at that famous Bombieri inequality and told me that it shouldn't be sharp for several polynomials.
Well, actually he said: "Seguro que es mucho más pequeña, es algo que se tiene que poder hacer". 
I want to thank him for pointing out this beautiful formula.
You were right, Carlos, it wasn't sharp.
I also want to acknowledge H\r{a}kan Hedenmalm for the discussions we had on this topic while I was writing the final version of this manuscript.% when visiting him on KTH.

%------------------------------------------------------------
%------------------------------------------------------------
%------------------------------------------------------------
%------------------------------------------------------------

\bibliography{Bombieri}{}
\bibliographystyle{plain}

%------------------------------------------------------------
%------------------------------------------------------------
%------------------------------------------------------------
%------------------------------------------------------------

\end{document}